\documentclass[11pt,twoside]{article}
\usepackage{amsmath,amssymb,amsthm}
\usepackage[top=1in, left=1.5in, right=1.5in, bottom=2in]{geometry}
\usepackage{palatino}
\usepackage{float}
\floatstyle{plaintop}
\usepackage{color}
\usepackage{relsize}
\usepackage{ifpdf}
\ifpdf
\usepackage[pdftex]{graphicx}
\else
\usepackage{graphicx}
\fi

\usepackage[toc,page]{appendix}

\usepackage[mathscr]{euscript}
\usepackage{subfigure}
\usepackage{enumerate}

\newtheorem{thm}{Theorem}[section] 
\newtheorem{lem}[thm]{Lemma} 
\newtheorem{prop}[thm]{Proposition} 
 
\newtheorem{conj}[thm]{Conjecture}
\newtheorem{definition}[thm]{Definition}
\newtheorem{remark}{Remark}

\newcommand{\cov}{\mathtt{Cov}}

\newcommand{\pat}{\mathcal{Q}}
\newcommand{\cp}{\mathcal{P}}
\newcommand{\dom}{\mathcal{C}}

\title{\LARGE  \bf Covering with Excess One: Seeing the Topology}
\author{Han Wang}
\date{}

\begin{document}
\maketitle
\begin{abstract}
	We have initiated the study of topology of the space of coverings on grid domains. The space has the following constraint: while all the covering agents can move freely (we allow overlapping) on the domain, their union must cover the whole domain. A minimal number $N$ of the covering agents is required for a successful covering of the domain. In this paper, we demonstrate beautiful topological structures of this space on grid domains in 2D with $N+1$ coverings, the topology of the space has the homotopy type of $1$ dimensional complex, regardless of the domain shape. We also present the Euler characteristic formula which connects the topology of the space with that of the domain itself.
\end{abstract}

\section{INTRODUCTION} 
\label{sec:introduction}
\subsection{BACKGROUND} 
\label{sub:background}
The spaces of coverings appear in many applications. Usually a (finite) family of balls are placed randomly to cover a metric space; such stochastic covering known as a \emph{coverage process} has been studied extensively (see e.g., ~\cite{MR973404}). In recent years, coverage problem in the context of sensor networks and robot motion planning has become an active research area; robot swarms (sensors) are deployed to cover a given domain (see~\cite{ExploreCoverage:DARS:10},~\cite{4739194}, etc.). If we consider each moving robot to be a labelled covering agent, we can formulate the space of coverings as follows: 

\begin{definition}
	The \textbf{covering configuration space} of $n$ labelled points on a topological space $X$ is defined as the space 
	\[
	\cov_{X}(n,\mathscr{F}):=\{(f_{1},f_{2},\dots,f_{n}), f_{i}\in \mathscr{F}:\bigcup_{i=1}^{n}\mathscr{M}_{f_{i}} = X \},
	\]
	where $\mathscr{F}$ is a topological space parametrizing subsets of $X$ as $\{\mathscr{M}_{f}\}_{f\in \mathscr{F}}$. 
\end{definition}

Recently using topological approach to study the covering configuration space, among other \emph{applied configuration spaces} (e.g., configuration space of hard spheres~\cite{Baryshnikov08022013}) starts to bring new insights into the subject. Understanding its topological features can be useful in many applications. The total Betti number of this space, for example, would provide lower bounds for the depth of the algebraic decision trees (which are used to decide membership in a semialgebraic set in theoretical computer science), according to Yao-Lov\'asz-Bj\"orner~\cite{MR1473048}. 

As our first attempt to understanding the covering configuration space for planar and higher dimensional domains, in the sequel we shall focus on grid domains which are formed by square plaques in 2D, all the coverings are of the same shape as cubical, in order to cope with the domain geometry in a nice fashion. By our assumption, there are $N+1$ such covering agents, and the domain can be covered by exactly $N$ of them. For simplicity whenever possible, we shall denote this covering configuration space as $\cov_{G_{N}}(N+1)$, where $G_{N}$ denotes the grid domain that can be covered precisely by $N$ unit square plaques.
\subsection{SIMPLE EXAMPLES} 
\label{sub:simple_examples}
The 1D scenario is simple: for a line segment (say, unit closed interval), consider covering it with $n$ balls of the same length. Y. Baryshnikov first studied this case, it turns out that the covering configuration
space is homotopy equivalent to
the skeleton of the \emph{permutahedron}~\cite{MR1311028} with $n$ vertices. More precisely, denote $\cov_{I}(n,r)$ as the covering configuration space for unit interval $I$ with $n$ balls of radii $r$, we have 

\begin{thm}[Y. Baryshnikov]
	\label{t:1D}
  \[ 
\cov_{I}(n,r)\simeq_h \mathtt{Skel}_{k} (\Pi_{n-1}),
\]
where $\Pi_{n-1}$ is the permutahedron with each vertice coresponding to a
permutation of $(1,\dots,n)$, $k=n-\lceil\frac{1}{2r}\rceil$ and $\lceil\frac{1}{2r}\rceil$ is the minimal number of $r$-balls to cover $[0,1]$. $k$ is the \emph{excess number}.
\end{thm}

Based on the theorem, we have the conclusion that for covering $I$ with excess one, we have 
\[
\cov_{I_{N}}(N+1)\simeq_h\mathtt{Skel}_{1} (\Pi_{N}).	
\]

For illustration purpose of Theorem~\ref{t:1D}, consider the following simplest nontrivial example: covering $I$ with three points. The reader may find this helpful for intuition. 

As $r$ varies, $\cov_I(3,r)$ changes. The trivial case happens when $r$ is sufficently large, say, $r=\frac12$, then $\cov_I(3,r)$ is contractible to a point. On the other extreme case, when $r$ is too small, $\cov_I(3,r)$ would be $\emptyset$, which means that there is no placement available to cover $I$. Suppose $r=\frac 16$, then we have to place the three balls side by side without any overlap. If unlabeled, this corresponds to one point in the covering configuration space. Therefore, $\cov_I(3,\frac 16)$ consists of $6$ isolated points. Suppose $r=\frac 14$, we are able to cover $I$ with exactly $2$ points, the third one can move freely on $I$, acting as the excess covering agent. When it overlaps with one or the other point (we shall call such placement as a critical configuration), the before-fixed one is free of being restricted in moving, and can move anywhere on $I$, as now it acts as the excess covering. Performing such permutation successively will lead to a rotation, hence $\cov_I(3,\frac 14)$ is homotopy equivalent to a hexagon, or simply a circle. This is exactly what Theorem~\ref{t:1D} tells us. On the other hand once isolated points become connected in the configuration space, as the number of excess covering agents increases from $0$ to $1$. 

\begin{figure}[h!]
\centering
\includegraphics[width=5cm]{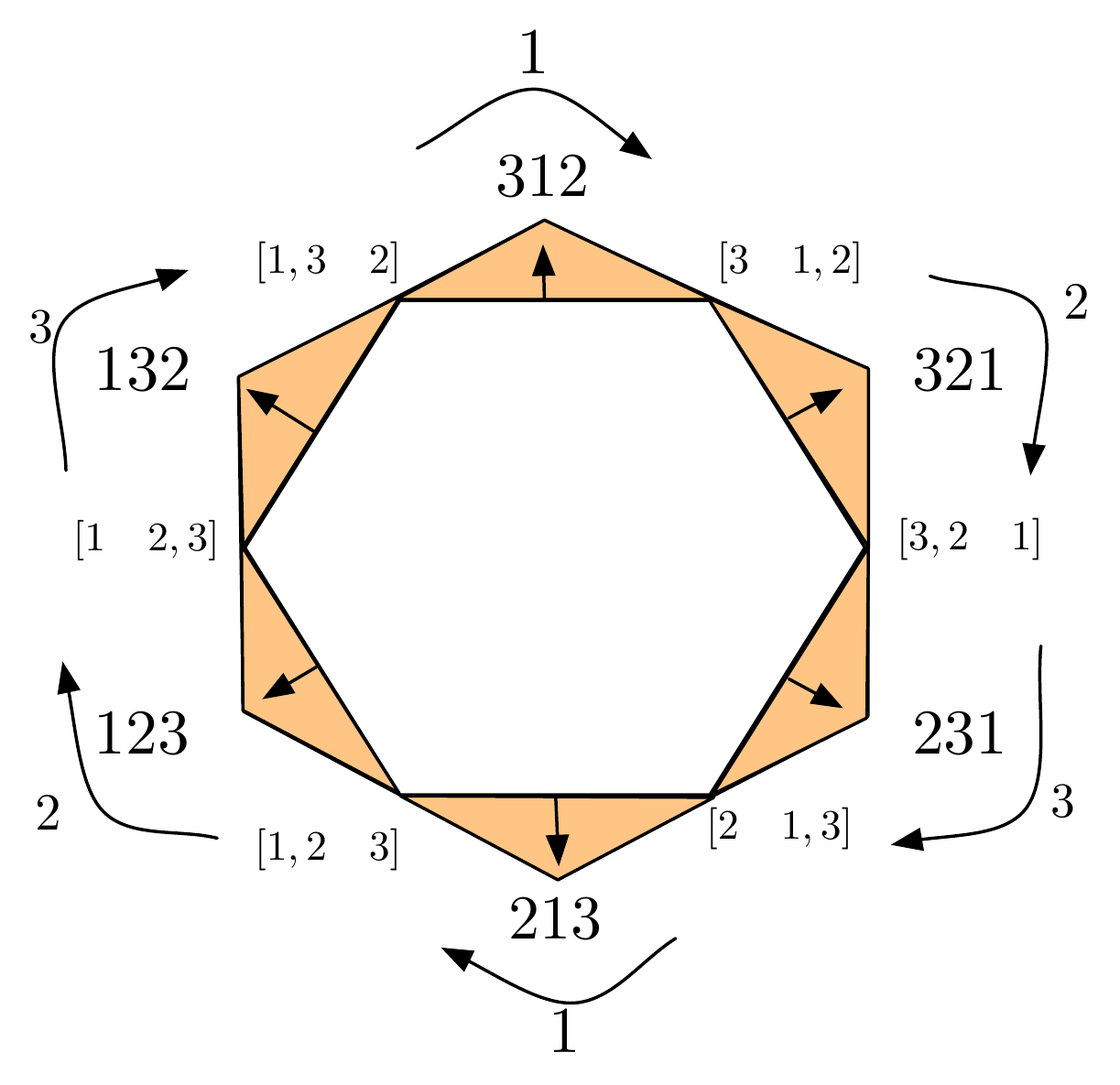}
\caption{Visualization of the covering configuration space for 3 points to cover $I$ while 2 is enough. The figure shows how the 1 skeleton of permutahedron $\Pi_{2}$ can be thought of as a homotopy equivalence of a hexagon with vertices labelled by critical configurations, plus 6 triangular cells glued on its edges.}
\label{fig:3pts}
\end{figure}

Here we have another point to make, which is a little technical but turns out to be very helpful in the proof we will present later: there is a one-to-one correspondence between an edge in $\mathtt{Skel}_{1}(\Pi_{2})$ and a critical configuration point. The visualization is given in figure~\ref{fig:3pts}. The inner hexagon has its vertices labelled by those critical configurations. Edge is formed when we move the excess agent to form another critical configuration. $\mathtt{Skel}_{1}(\Pi_{2})$ (the bigger hexagon) can be recovered first by the expansion of the inner hexagon and then by deformation retracting to it. We shall make use of this `dual' viewpoint as it helps us identify the critical configurations while the skeleton of a permutahedron would not naturally do so.   

As a more interesting example, the visualization of $\cov_{I_{3}}(4)$ is given in figure~\ref{fig:4pts}.

\begin{figure}[h!]
\centering
\includegraphics[width=8cm]{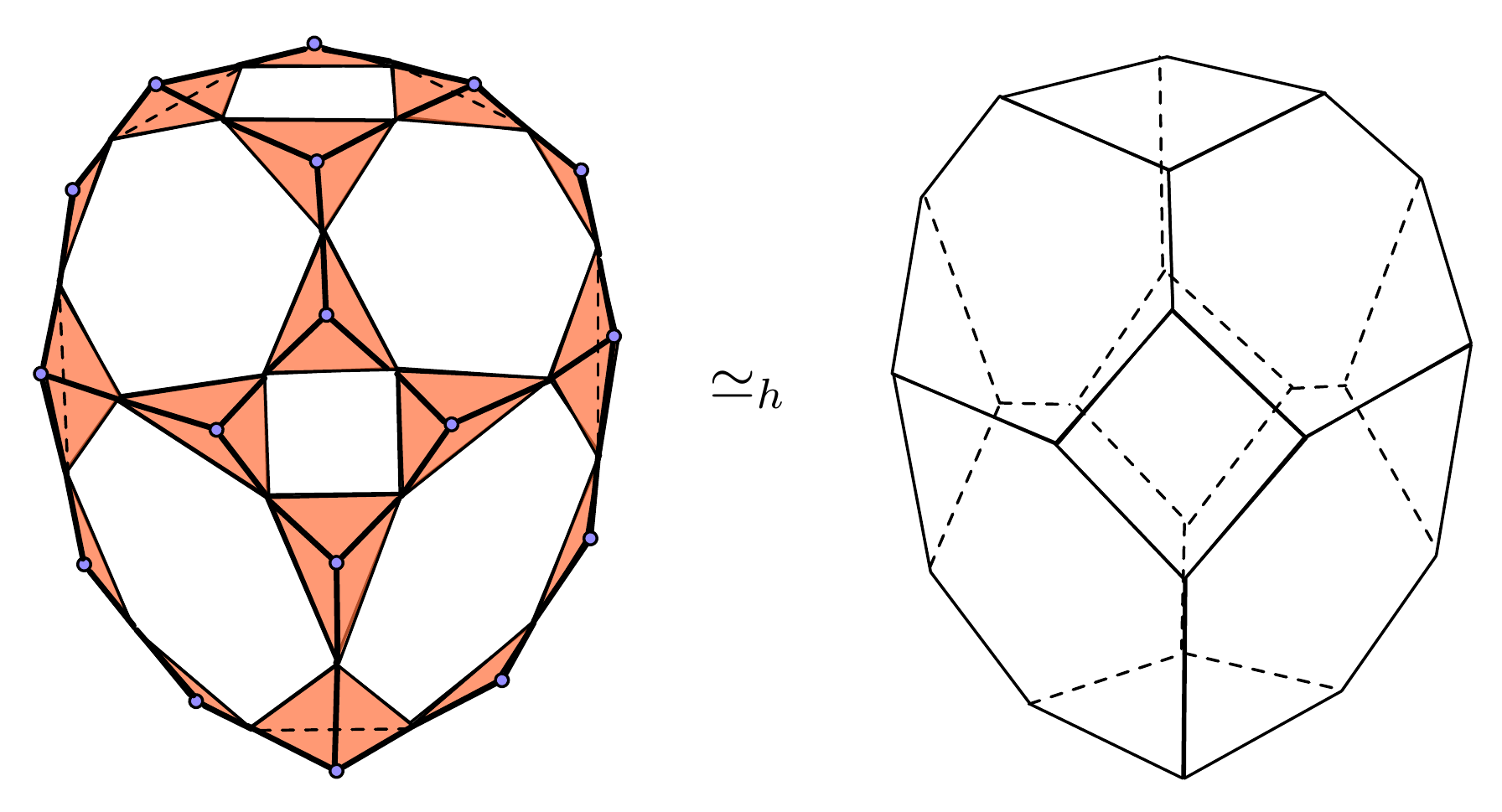}
\caption{Visualization of the covering configuration space for 4 points with radius $\frac 16$ to cover $I$ (in this case exactly three points cover $I$ just right), on the right is the 1 skeleton of the permutahedron $\Pi_{3}$, on the left is the `dual' view.}
\label{fig:4pts}
\end{figure}

Now consider a 2D grid domain $G_{2\times 2}$ formed by 2 by 2 square plaques. The excess agent have the freedom of moving either vertically or horizontally, with other four covering agents fixed. Its free patrolling region $\pat$ is shown in figure~\ref{fig:2by2}.

\begin{figure}[h!]
\centering
\includegraphics[width=4cm]{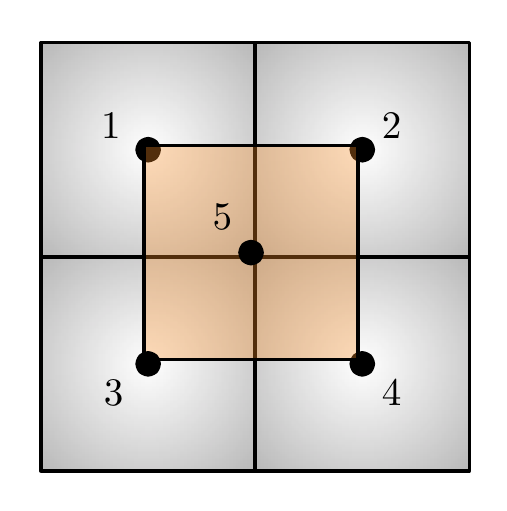}
\caption{Covering a 2 by 2 grid domain with 5 balls $B_{\infty}(x_{i},\frac12)$, the central square is the region $\pat$ where number 5 can move freely.}
\label{fig:2by2}
\end{figure}

Note that restricting the excess's move on one edge of $\partial\pat$ becomes the problem of covering 1D line segments, we know $\cov_{I}(3,1/2)$ has the homotopy type of a hexagon, moving on one edge of $\partial \pat$ traverses one edge of the hexagon, so $\partial \pat$ can be glued with $\mathtt{Skel}_{1}(\Pi_{2})$ along that edge. Let's identify $\pat$ with a $2$ by $2$ labeling matrix $C$, namely, if $c_{ij}\in\{1,2,3,4,5\}$ for all $i,j=1,2$, then $C$ is identified with $\pat$; if $C$ has only three elements from the set $\{1,2,3,4,5\}$ specified, then $C$ is identified with a vertex of $\pat$; if $C$ has only two elements from the set $\{1,2,3,4,5\}$ specified in a row or in a column, then $C$ is identified with a hexagon with one edge on $\pat$. Tracking the labeling $C$ for cells of different dimensions in $\pat$ tells us how $\pat$ is glued in with other parts locally, see figure~\ref{fig:2by2glue}, which forms the subcomplex of $\cov_{G_{2\times 2}}(5)$.

\begin{figure}[h!]
\centering
\includegraphics[width=8.7cm]{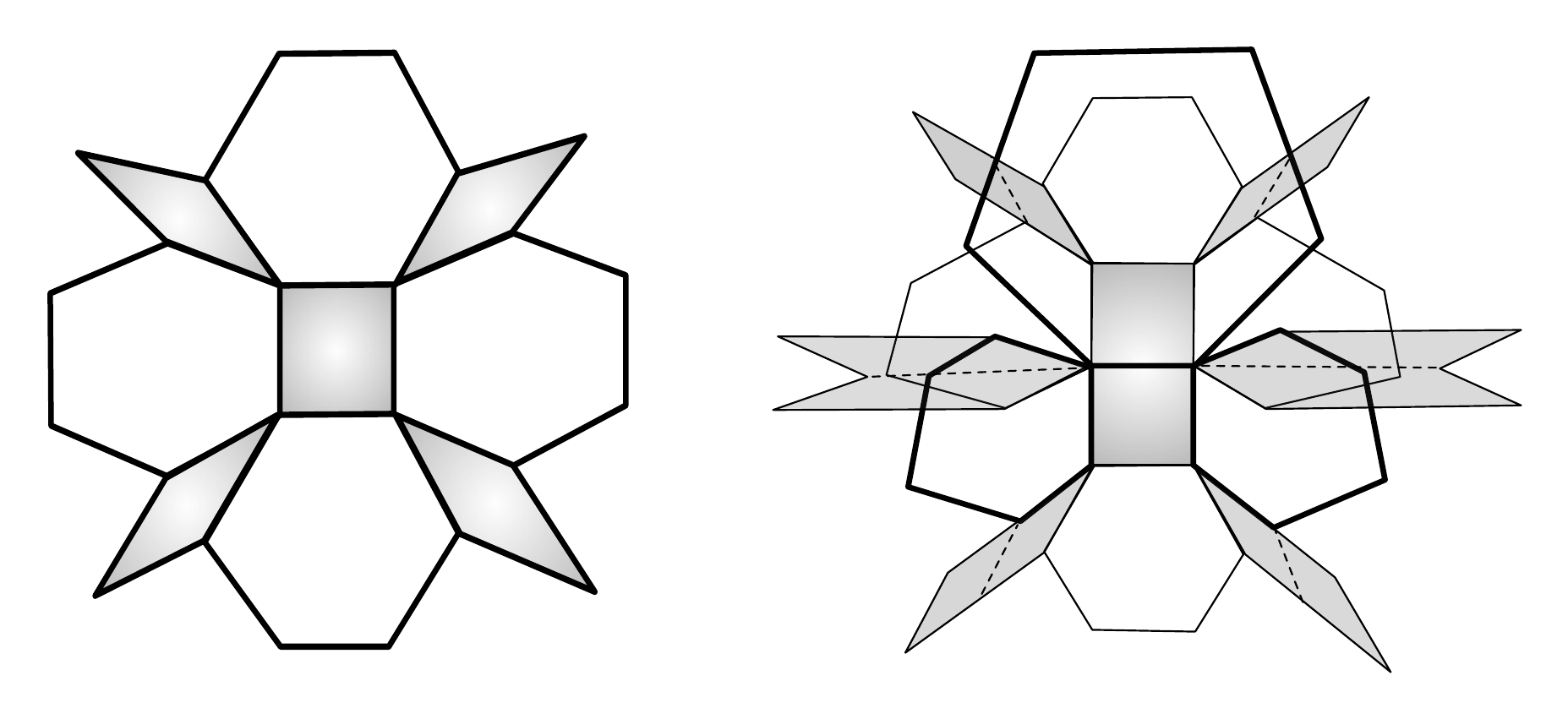}
\caption{The figure on the left illustrates how $\pat$ for $G_{2\times 2}$ is glued with $\mathtt{Skel}_{1}(\Pi_{2})$ and free patrolling regions of different excess covering by choice. Similar visualization can also be done for 3D grid domain $G_{2\times 2\times 2}$. On the right is the gluing for $G_{2\times 3}$.}
\label{fig:2by2glue}
\end{figure}

Since $\pat$ has free face in $\cov_{G_{2\times 2}}(5)$, therefore, with the removal of 2-cells, the complex is collapsible to a 1 dimensional complex. We only need to enumerate $\pat$ and hexagons. Therefore, we have $5!$ $\pat$ and $4\times P(5,2)$ hexagons. Note that each edge is shared by one $\pat$ and one hexagon, the total number of edge should be $(4\times 5!+6(4\times P(5,2)))/2$; each vertex is shared by two hexagons (or alternatively, two $\pat$'s), the total number of vertices should be $4\times5!/2$. Thus, the Euler characteristic should be 
\[
	\chi(\cov_{X_{4}}(5))=\#(\pat)-\#(\text{edges})+\#(\text{vertices})=-5!\;.
\]

\subsection{MAIN RESULTS} 
\label{sub:main_results}
Covering with excess one becomes more complicated in higher dimensions, as the excess covering can accordingly move freely in higher dimensional regions. In the sequel we assume the grid domain $G_{A}$ is connected in 2D and is formed by unit square plaques. The area of $R_{A}$ is $A$; this implies we can cover $G_{A}$ with exactly $A$ unit square plaques. 

We have proved the following: 

\begin{thm}\label{t:collapsible}
	Suppose the grid domain $G_{A}$ can be covered precisely by $A$ unit square plaques, then $\cov_{G_{A}}(A+1)$ is homotopy equivalent to $1$ dimensional complex.
\end{thm}

Since the essential topological feature of a $1$-dimensional complex is determined by the Euler characteristic, we have also showed that

\begin{thm}\label{t:Euler}
	The Euler characteristic of $\cov_{G_{A}}(A+1)$ is
	\[
	\chi(\cov_{G_{A}}(A+1))=(\chi(G_{A})-A/2)(A+1)!,
	\]
	where $\chi(G_{A})$ is the Euler characteristic for the domain $G_{A}$.
\end{thm}

\section{BASICS} 
\label{sec:basics}
Before we give the proof, we begin with some definitions in this section for preparation. 

\begin{definition}
	A \textbf{polyhedral complex} $\cp$, as a special kind of cell complex, is a collection of convex polytopes such that
	\begin{enumerate}
		\item every face of a polytope in $\cp$ is a polytope itself in $\cp$;
		\item the intersection of any two polytopes in $\cp$ is a face of each of them.
	\end{enumerate}
\end{definition}
Later we will see that $\cov_{G_{A}}(A+1)$ has the structure of a polyhedral complex. 

\begin{definition}
	Let $\cp$ be a polyhedral complex. A \textbf{free face} of $\cp$ is some $\tau\in\cp$ such that there exists one and only one $\sigma\in \cp$ with $\tau\subset\sigma$. In particular, $\dim\tau<\dim\sigma$. An \textbf{elementary collapse} of $\cp$ is the removal of the pair $(\tau,\sigma)$ when $\dim\tau=\dim\sigma-1$.
\end{definition}

\begin{definition}
	A polyhedral complex $\cp_{1}$ \textbf{collapses} to another polyhedral complex $\cp_{2}$ if there exists a sequence of elementary collapses, we write $\cp_{1}\searrow\cp_{2}$. An \textbf{expansion} is the operation inverse to the collapse. We say $\cp_{1}$ is the expansion of $\cp_{2}$ and write $\cp_{2}\nearrow\cp_{1}$.
\end{definition}

The following proposition is well-known:
\begin{prop}\label{prop:collapse}
	A sequence of collapses yields a strong deformation retraction, in particular, a homotopy equivalence.
\end{prop}

In proving the 2D case, the following labeling notations are given for our convenience.

Since each plaque in the 2D domain must be covered by at least one point; in order to specify particular covering configurations and keep tract of them, we use $\mathcal{C}_{N}$ to define the labeling set
\[
	\{C=\left[\begin{smallmatrix}
	c_{11} & c_{12} & \ldots & c_{1n}\\
	\vdots & \vdots & \ddots & \vdots\\
	c_{m1} & c_{m2} & \ldots & c_{mn}
	\end{smallmatrix}\right]:c_{ij}\subsetneq [N],\; \bigsqcup_{ij} c_{ij}=[N]\},
\]
where $\bigsqcup$ stands for disjoint union and $[N]$ denotes the set $\{1,\dots,N\}$. 

Working on a grid domain $G_{A}$ with arbitrary shape, we shall assume $G_{m\times n}$ is the minimal rectangular grid domain that covers $G_{A}$. For the pair of $i,j$ corresponding to square plaque outside $G_{A}$, we simply ignore $c_{ij}$ in that position. Other $c_{ij}$'s corresponding to positions of different plaques in $G_{A}$ will be \emph{active}.

Reading off the active $c_{ij}$ tells us which point(s) is/are used to cover the plaque on $G_{A}$ at the $i$th row and the $j$th column of the rectangular grid domain $G_{m\times n}$. For example, when $m=2,n=3,N=7$, $\left[\begin{smallmatrix}
\{1\} & \{3\} & \{2\}\\
\{4\} & \{5\} & \{6,7\}
\end{smallmatrix}\right ]$ stands for a covering configuration that points $1,2,3,4,5$ are taking care of their own plaques, while $6,7$ together cover the plaque in the lower right corner. $C=\left[\begin{smallmatrix}
c_{11} & c_{12} & c_{13}\\
\{4\} & \{5\} & \{6\}
\end{smallmatrix}\right ]$ with the active $c_{11},c_{12},c_{13}$ unspecified corresponds to the situation that we have points $4,5,6$ covering the second row of plaques, and the first row of plaques are covered up by points from $\{1,2,3,7\}$; in this case, Theorem~\ref{t:1D} indicates that the minimal covering configuration space traversing $\{\left[\begin{smallmatrix}
c_{11} & c_{12} & c_{13}\\
\{4\} & \{5\} & \{6\}
\end{smallmatrix}\right ]\}$ has the homotopy type of $\mathtt{Skel}_1(\Pi_3)$. 

Given a covering configuration $\mathsmaller{C=\left[\begin{smallmatrix}
	a_{11} & a_{12} & \ldots & a_{1n}\\
	\vdots & \vdots & \ddots & \vdots\\
	a_{m1} & a_{m2} & \ldots & a_{mn}
	\end{smallmatrix}\right]}$, where each $a_{ij}$ is a singleton except for the inactive ones, the free region $\pat_{p}$ for the excess covering agent $p\notin a_{ij}$ is defined as follows:
	
	\begin{definition}
		The \textbf{free patrolling region} $\pat_{p}$ for $p$ is a $2$ dimensional polyhedral complex such that
		\begin{itemize}
			\item the $0$-cells $v_{(i,j)}$ corresponds to fixed covering point $a_{ij}$.
			\item the $1$-cells are edges $\{v_{(i,j)},v_{(i,j+1)}\}$ and $\{v_{(i,j)},v_{(i+1,j)}\}$ that connect neighboring covering points.
			\item the $2$-cells are unit square plaques $\{v_{(i,j)},v_{(i,j+1)},v_{(i+1,j)},v_{(i+1,j+1)}\}$ that are glued in along its boundary edges.
		\end{itemize}
	\end{definition}
\begin{remark}[On the labeling]
	\begin{enumerate}
		\item Note that there are $(A+1)!$ such free patrol regions, since the excess one is identified whenever we have made a choice of the fixed covering points. It is natural for us to use 
		$\mathsmaller{\left[\begin{smallmatrix}
		a_{11} & a_{12} & \ldots & a_{1n}\\
		\vdots & \vdots & \ddots & \vdots\\
		a_{m1} & a_{m2} & \ldots & a_{mn}
		\end{smallmatrix}\right]}$ 
		to label $\pat_{p}$.
		\item We shall also use 
		$\mathsmaller{\left[\begin{smallmatrix}
		a_{11} & a_{12} & \ldots & a_{1n}\\
		\vdots & \vdots & a_{ij}\cup \{p\} & \vdots\\
		a_{m1} & a_{m2} & \ldots & a_{mn}
		\end{smallmatrix}\right]}$ 
		to label the $0$-cell $v_{(i,j)}$. This allows us to identify different $0$-cells on different $\pat_{p}$'s.
	\end{enumerate}
\end{remark}	
	In principle, we can use the devised labeling notations to label all the cells in the complex for the covering configuration space. The right way of gluing for the covering configuration space means a consistent labeling from cells of 0 dimension to the working dimension, which shall be made clear in the proof.

Finally, the \emph{crossings} of $\pat_{p}$ are maximal horizontal (resp. vertical) lines traversing the horizontal (resp. vertical) edges. Let the \emph{length} of a crossing be the number of vertices on it.
\section{PROOF} 
\label{sec:proof}
\textbf{Construction of $\cov_{G_{A}}(A+1)$}:

Our construction of $\cov_{G_{A}}(A+1)$ can be described as a \emph{gluing} process. The basic idea is to glue the $2$ dimensional complex of free patrol region with the $1$ skeleton of permutahedrons along the crossing lines. We shall take advantage of Theorem~\ref{t:1D}, but in order to facilitate the glueing, we first need to present the  modification step for $\mathtt{Skel}_{1}(\Pi_{n-1})$:

We denote $\mathtt{Skel}_{1}(\Pi_{n-1})$ as a polyhedral complex $\cp$. In order to `patrol' on the $1$ skeleton of permutahedrons, we construct the modification $\cp'$ as follows:
\begin{enumerate}
	\item We first have a \emph{barycentric subdivision} of $\cp$ denoted as $S\cp$. The mid-points on each edge of $\cp$ are $0$-cells of $S\cp$ now.  
	\item We construct the expansion of $S\cp$ by adding all pairs $(e,t)$ with $e$ and $t$ satisfying the following conditions:
	\begin{itemize}
		\item For two edges of $\cp$ which are connected to a vertex $v$ of $\cp$, $e$ is the edge connecting the mid-points on them. 
		\item $t$ is triangular $2$-cell that contains $e$ and the vertex $v$. 
	\end{itemize} 
\end{enumerate}
Denote the resulting $2$ dimensional polyhedral complex as $\cp'$. Since $\cp'\nearrow S\cp$, we immediately have 

\begin{lem}
	$\cp'$ is homotopy equivalent to $\cp$.
\end{lem}

This visual change yields advantage in the glueing process. This is because edges in $\cp$ corresponds to adjacent transpositions. The mid-points on edges of $\cp$ can be used to represent critical covering configurations that two points (the permuted pair) are overlapped. The edges in $\cp'$ whose boundary are those mid-points (denoted as $e_{v_{i}v_{j}}$) shall be treated as patrolling paths from one critical covering configuration to another. Indeed, if we look at all the edges in $\cp$ whose boundary contains the vertex $v=123\cdots n$, they (or alternatively, the midpoints) can be represented as $[(12)3\cdots n]$, $[1(23)\cdots n]$, $\dots$, $[123\cdots (n-1,n)]$, the points in bracket together cover a fraction of line segment while the others cover one fraction by themselves. Hence these vertices form a set of critical covering configurations for the line segment.

There are two types of patrol paths represented by $e_{v_{i}v_{j}}$:
\begin{itemize}
	\item If we take $v_{1}=[(12)3\cdots n]$, $v_{2}=[1(23)\cdots n]$, $e_{v_{1}v_{2}}$ represents the path for point $2$ patrolling from $1$ to $3$. We call such $e_{v_{1}v_{2}}$ representing a \emph{single shift}, denote the set of all single shift edges as $E_{1}$.
	\item If we take $v_{1}=[(12)3\cdots n]$, $v_{2}=[12(34)\cdots n]$, $e_{v_{1}v_{2}}$ represents the patrol path for both $2$ and $3$ synchronously moving to the right from configuration $v_{1}$ to $v_{2}$. This involving more than one point movement, hence we call such $e_{v_{1}v_{2}}$ representing a \emph{swarm shift}, denote the set of all swarm shift edges as $E_{2}$. 
\end{itemize}

When the excess agent travels along crossings of the free patrol region $\pat$, locally it travels on lines from one critical covering configuration to another, therefore, edges representing single shift in $\cp'$ can be identified with edges in $\pat$.

Now we can describe the glueing process for $\cov_{G_{A}}(A+1)$ as the polyhedral complex $\mathcal{K}$:

\begin{enumerate}
	\item Start with the $1$-skeleton of $\pat_{p}$ for an arbitrary patroller $p$. Note that there are $(A+1)!$ such complexes. On $\pat_{p}$ we can list all the crossings. For the $i$th crossing on $\pat_{p}$ with length $n$, we have $\mathtt{Skel}_{1}(\cp')$, where $\cp'$ is the expansion of $S\mathtt{Skel}_{1}(\Pi_{n})$.  The edges on $\mathtt{Skel}_{1}(\cp')$ that represent single shift shall be identified with edges on the crossing of $\pat_{p}$. Denote the new space as $\mathcal{K}_{p}^{i}=\mathtt{Skel}_{1}(\pat_{p})\cup_{E_{1}}\mathtt{Skel}_{1}(\cp')$.
	\item We glue $\mathcal{K}_{p}^{i}$ together for different $i$'s by identifying $\mathtt{Skel}_{1}(\pat_{p})$ that is common to them. 
	\item We glue $\mathcal{K}_{p}$ together for different $p$'s by identifying the vertices that have the same label. 
	\item Finally, glue in the $2$-cells of $\pat_{p}$ and $2$-cells of $\cp'$ along their boundaries. The resulting polyhedral complex $\Delta$ is the covering configuration space we are after. 
\end{enumerate}

The reader can refer figure~\ref{fig:2by2glue} for some illustrations. 

\noindent\textbf{Collapsibility:}

The important observation of the existence of free edges leads us to the following collapsibility:  
\begin{lem}\label{lem:collapsible}
	$\mathcal{K}$ is collapsible to a $1$-dimension complex.	
\end{lem}

\begin{proof}[Sketch of the proof:]
	We prove by giving the sequence of collapse: first for triangular $2$-cells in $\cp'$: the expansion of 1 skeleton of $\Pi_{n}$'s, where $n$ is the length of different crossings in $\pat_{p}$.  When $\cp=\mathtt{Skel}_{1}\Pi_{n}$ has less than $3!$ vertices, we can take the collapsible pair (e,t) with e being the edge with one of its boundary point at a vertex of $\cp$. For $\cp$ having at least $4!$ vertices, the expansion $\cp'$ has at least one collapsible pair $(e,t)$ with $e$ not shared by other $2$-cells from either $\cp'$ or $\pat$. Such $e$ represent a swarm shift on $\pat$. After the removal of $(e,t)$, edges belonging to the subdivision of $\cp$ on $t$ become free. Therefore, we have a sequence of elementary collapses for all triangular $2$-cells in $\cp'$. 
	
	Note that every $\pat$ is identified with some excess agent while the other covering points are fixed. Permutation only happens when it is overlapped with some other covering point, therefore, $\pat_{p}$ for excess agent $p$ is identified with $\pat_{q}$ for excess agent $q$ by a vertex if and only if they are overlapped at the vertex; the edges on $\pat_{p}$ are never shared by 2-cells from $\pat_{q}$, for $q\neq p$. Therefore, working from plaques with free edges on $\pat_{p}$ to plaques surrounded by other plaques, we can have a sequence of collapses for $\pat_{p}$, leading to 1 dimensional polyhedral complex. 
\end{proof}

Theorem~\ref{t:collapsible} instantly follows. We put the complete proof (using partial matching arguments) in the appendix. 

Next, one derives the Euler characteristics by counting the number of cells in different dimensions. 

Note that 
$\mathsmaller{\pat_{p}=\left[\begin{smallmatrix}
	a_{11} & a_{12} & \ldots & a_{1n}\\
	\vdots & \vdots & \ddots & \vdots\\
	a_{m1} & a_{m2} & \ldots & a_{mn}
	\end{smallmatrix}\right ]}$ and $\mathsmaller{\pat_{q}=\left[\begin{smallmatrix}
	a_{11}' & a_{12}' & \ldots & a_{1n}'\\
	\vdots & \vdots & \ddots & \vdots\\
	a_{m1}' & a_{m2}' & \ldots & a_{mn}'
	\end{smallmatrix}\right]}$ are glued by a vertex if and only if $a_{ij}=a_{ij}'$ 
	for all but one pair of $i$ and $j$.

\noindent\textbf{Counting:}

Combining all the above lemma, we set forward to counting the total number of cells of $\mathcal{K}$ in different dimensions. 

Suppose $G_{A}$ is a 2 dimensional domain with Euler characteristic $1-g$, it has the homotopy type of a disk with $g$ points removed. We have the following lemma first.

\begin{lem}
	Suppose there are $n_{i}$ crossings of length $i$ in total on the free patrolling region $\pat$ and the area of $\pat$ is K, $\chi(G_{A})=1-g$, then we have 
	\[
		K=1-A+\sum_{i}n_{i}(i-1)-g.
	\]
\end{lem}

\begin{proof}
	Our proof is carried out in two stages: first consider when $G_{A}$ is contractible. This would allow us to build $G_{A}$ in the following way:
	\begin{enumerate}
		\item beginning from the top row, we construct $G_{A}$ from left to right by adding unit square blocks one by one.
		\item for the next row building, since $G_{A}$ is contractible, we can always pick one block whose boundary is attaching the first row. We then extend the row by adding blocks to its left and right. Then we pick another block in this row attaching to the first row and extend it again. Repeat until we finish building the second row.
		\item repeat the above steps row by row until we complete the last row building. 
	\end{enumerate}

	During the construction, we observe how $K$ is changing when $A$ and the $\sum_{i}n_{i}(i-1)$ terms change. When building the first row of $\dom$, adding one block, we have $\Delta A=1$, $n_{i}=1$ and $\Delta i=1$, $K=0$. Note that with $\Delta A=1$, $\Delta K=1$ if and only if a cross emerges in $\pat$. That implies $\Delta \sum_{i}n_{i}(i-1)=2$, by induction, we have
	\[
		K=1-A+\sum_{i}n_{i}(i-1),
	\] 
	when $G_{A}$ is contractible.
	
	Next, suppose we have holes in $G_{A}$. We again take the constructional tactic. The base step is when 
	\[
		G_{A}'=G_{A}-\sigma_{s_{1}},
	\]
	where $\sigma_{s}$ stands for a solid square and $G_{A}$ is contractible, $\sigma_{s_{1}}$ is in the interior of $G_{A}$. We have $\Delta A=-1$, $\Delta K=-4$, and since both the horizontal and vertical crossings become two shorter crossings, $\Delta \sum_{i}n_{i}(i-1)=-4$. Therefore, by induction, if we remove $g$ isolated blocks in the interior of $G_{A}$, we have 
	\[
		K=1-A+\sum_{i}n_{i}(i-1)-g.
	\]
	When enlarging the holes, this equality remains. To see this, we consider all the possible situations when a generic hole (a missing block in the interior of $G_{A}$ is enlarged. Suppose a sequential removal of $\sigma_{s_{1}},\dots,\sigma_{s_{n}}$ leads to an enlarging hole in $G_{A}$. $\Delta A=-1$ with $\sigma_{s_{i+1}}$ removed.
	Therefore, if $\sigma_{s_{i}} \cap\sigma_{s_{i+1}}$ is an edge, $\Delta K=-2$, $\Delta \sum_{i}n_{i}(i-1)=-3$; if $\sigma_{s_{i}}\cap\sigma_{s_{i+1}}$ is a vertex, $\Delta K=-3$, $\Delta \sum_{i}n_{i}(i-1)=-4$. In both cases, $\Delta K=-\Delta A+\Delta\sum_{i}n_{i}(i-1)$. We are done.
\end{proof}

Now we can prove Theorem~\ref{t:Euler}.

Suppose $\pat$ has $n_{i}$ crossings with length $i$, permutation happens only on the crossing lines. Enumerating the number of cells in different dimensions is similar to what we have done in last section:

\begin{equation}
	\#(\sigma_{s})=(A+1)!K,
\end{equation}
where $K$ is the number of $\sigma_{s}$ in $\pat$.

\begin{equation}
	\#(\sigma_{t})=\sum_{i} P(A+1,A-i)n_{i}\frac {i(i-1)}2(i+1)!
\end{equation}
The triangular 2-cells all come from the expansion of permutahedrons.

\begin{equation}
	\begin{aligned}
		\#(e)=&\mathsmaller{\sum\limits_{i}(n_{i}-1)(A+1)!+\sum\limits_{i} P(A+1,A-i)n_{i}i(i+1)!}\\
		&\mathsmaller{+\sum\limits_{i} P(A+1,A-i)n_{i}\left(\frac {i(i-1)}2-\left(i-1\right)\right)(i+1)!}
	\end{aligned}	
\end{equation}

The first part comes from edges on $\pat$, the second and third part comes from edges on the expansion of permutahedrons that are not glued with those on $\pat$. They belong to swarm shift and edges connecting to the vertices of permutahedrons respectively.

\begin{equation}
	\#(v)=\frac A2(A+1)!+\sum_{i} P(A+1,A-i)n_{i}(i+1)!
\end{equation}
The first part comes from vertices on $\pat$, each is shared by two $\pat$ of different patrollers. The second part comes from vertices of the permutahedrons.

Therefore, combining equations (1)-(4), we have 
\begin{equation*}
	\begin{aligned}		\chi(\cov_{G_{A}}(A+1))&=\#(v)-\#(e)+\#(\sigma_{s})+\#(\sigma_{t})\\&=\left(K+A/2-\sum_{i}n_{i}\left(i-1\right)\right)(A+1)!.
	\end{aligned}
\end{equation*}

Together with the lemma, we are done with the proof for 2D domain $G_{A}$. 

\section{GENERALIZATION TO 3D} 
\label{sec:generalization_to_3d}
Theorem~\ref{t:collapsible} can be generalized to some 3D grid domains with little modification. 3-cells in the free patrolling region $\pat_{p}$ for the excess agent $p$ can only share a vertex for 3-cells in the free patrolling region $\pat_{q}$, where $q\neq p$. again, the complex for the covering configuration space is collapsible to a 1 dimensional complex, as long as the grid domain itself is collapsible to a 1 dimensional complex. 
The Euler formula for the 3D case should work similarly as Theorem~\ref{t:Euler}. We have worked on some specific situations: for example, when $G_{A}=G_{a\times b\times c}$, we have 
\[
\chi(\cov_{G_{A}}(A+1))=(1-A/2)(A+1)!, 
\]
	where $A=abc$.
When $G_{A}$ has $g$ cavities, it has nontrivial 2 dimensional homology, we have
\[
\chi(\cov_{G_{A}}(A+1))=(1+g-A/2)(A+1)!.
\]
We conjecture that
\begin{conj}
	The Euler characteristic of $\cov_{G_{A}}(A+1)$ is
	\[
	\chi(\cov_{G_{A}}(A+1))=(\chi(G_{A})-A/2)(A+1)!,
	\]
	where $\chi(G_{A})$ is the Euler characteristic for the domain $G_{A}$ in both 2D and 3D. 
\end{conj}
But $\cov_{G_{A}}(A+1)$ now is not necessarily homotopy equivalent to 1 dimensional complex in this general setting.
\section{CONCLUSIONS} 
\label{sec:conclusions}
We have studied the problem of covering grid domains with excess one. In this situation, we are able to find a stunningly simple visualization of the topology of the covering configuration space. The Euler formula we find establishes a straight forward relationship between the topology of the space of coverings and the geometry as well as the topology of the working domain. By assuming simple geometric shape of the domain and a good match with the covering agents, We have managed to utilize combinatorial methods to find the important topological feature of the covering configuration space. One naturally ask the question: what can we say about the topology with excess number more than one? To see the topology would become more complicated, but can we still find some import topological quantity, say, the total Betti number? We can also ask questions about the unlabeled covering configuration space. What can we say about the symmetric group action on the homology group of the covering configuration space? We shall address these issues in follow-up papers.
\section{Acknowledgement} 
\label{sec:acknowledgement}
The author wants to thank Prof. Y. Baryshnikov for many inspiring discussions. 
\begin{appendices}
\section{Collapsibility} 
\label{sec:collapsibility}
Constructing a sequence of collapses in standard text can be well described using partial matching defined on a poset (also called discrete vector field by R. Forman, see~\cite{MR1612391}), the latter is actually more general than just the elementary collapses. To begin with this standard treatment, we recognize a polyhedral complex $\cp$ as a poset consisting of all the faces and whose partial order relation is the \emph{covering relation} on the set of faces, namely, $x\prec y$ if $x\subset y$ and there is no $z\in \cp$ such that $x\subset z\subset y$.

To describe the structural collapses, we follow the definitions of partial matching from D. Kozlov in~\cite{MR2361455}. 
\begin{definition}
	A \textbf{partial matching} on a poset $P$ is a subset $M\subset P\times P$ such that 
	\begin{enumerate}
		\item $(a,b)\in M$ implies $a\prec b$;
		\item each $\sigma\in P$ belongs to at most one element in $M$.
	\end{enumerate}
	A partial matching is called \textbf{acyclic} partial matching $(a_{i},b_{i})\in M$ if there does not exist the following nontrivial closed path
	\[
		a_{1}\prec b_{1}\succ a_{2}\prec b_{2}\succ\cdots\succ a_{n}\prec b_{n}\succ a_{1}
	\]
	with $n\geq 2$ and all $b_{i}$ being distinct.
\end{definition}
Note that a single pair of matching is always acyclic, but they may not be an elementary collapsable pair. Still, collapsing the matched elements in an acyclic partial matching will not change the homotopy type of the underlying space.  

Our goal is to prove the following lemma: 
\begin{lem}
	any $2$-dimension cell in the polyhedron complex $\mathcal{K}$ of our covering configuration space belongs to an element (collapsing pair) in some acyclic matching on $\cp$.
\end{lem}

Given a 2D grid domain $G_{A}$, possibly with holes, we can fix all the other covering agent while allowing one to freely patrol, this enables us to label a subcomplex $\pat_{p}$ of the covering configuration space. All the $2$-cells in $\cov_{G_{A}}(A+1)$ have exactly two geometric shapes:
\begin{enumerate}
	\item square plaques $\tau_{s}$: due to planar movement of one single covering agent. They all come from $\pat_{p}$ for some $p$.
	\item triangular plaques $\tau_{t}$: due to one dimensional (horizontally or vertically) rotations of multiple covering agent. They all come from the expansions of certain permutahedrons glued with $\pat_{p}$ along its crossings.  
\end{enumerate}
We describe the matching as follows: for different $\tau_{t}$'s, match each with $\sigma\prec \tau_{t}$ satisfying $\sigma\notin \partial \pat_{p}$ and preferably choose the $\sigma$ which is free (not shared by another $\tau_{t}$); if $\sigma\notin \partial \pat_{p}$ is not free, we can choose one arbitrarily while avoiding repetitive matching it with other $\tau_{t}$'s. 

For different $\tau_{s}$'s, we match them in the following way:
\begin{enumerate}[$\bullet$]
	\item We say a $\tau_{s}$ is on the boundary if $\tau_{s}$ has one edge $\sigma$ on $\partial \pat_{p}$ for some $p$, match $\tau_{s}$ with $\sigma$. Denote the collection of all the boundary ones as the set $M$. 
	\item  For other $\tau_{s}$, if $\tau_{s}\cap \tau\neq \emptyset$, for some $\tau\in M$, then they must intersect at some edge $\sigma$, match $\tau_{s}$ with $\sigma$. If they intersects at more than one edges, we can pick up an arbitrary one to match with $\tau_{s}$. Update $M$ with all the matched $\tau_{s}$'s.
	\item Repeat the last step until $M$ contains all $\tau_{s}$'s.  
\end{enumerate}

This partial matching is well-defined and is acyclic. Indeed, every $2$-cell can be matched with their incidental edges without repetition. This relies on the following observation: for $\tau_{t}$, each has at least two edges which are not part of $\partial \pat_{p}$; and at most two $\tau_{t}$'s share an edge. As for $\tau_{s}$, each has 4 candidate edges to be matched with, and at most two $\tau_{s}$'s share an edge. Hence each $2$-cell can be matched with one of their incidental edges without repetitive matching with other $2$-cells. The recursive way of matching all $\tau_{s}$ is valid as essentially any $\tau_{s}$ is path connected to one on the boundary. 

Next we show it is acyclic. Suppose there exists a sequence $\sigma_{0},\cdots,\sigma_{t}$ such that all $\sigma_{i}$ are different except $\sigma_{t}=\sigma_{0}$, each $\sigma_{i}$ is matched with some $2$-cell, they form a nontrivial closed path. Then we claim $\sigma_{0}$ cannot be an incidental edge of any $\tau_{t}$, the triangular $2$-cells. Indeed, first $\sigma_{0}$ cannot be a free edge of some $\tau_{t}$, since otherwise $\tau_{t}\succ \sigma_{0}$ has to appear in the matching at least twice, which is not allowed by the definition of partial matching. Secondly, for other $\sigma_{0}\notin \partial \pat_{p}$, suppose it is shared by two triangular $2$-cells, to make it a closed path, we have to make the two cells the first and the last one respectively in the closed path, this is only possible when we include all the triangular cells which share a single vertex (vertex of the permutahedron) in the closed path, thus all $\tau_{t}$'s has to be matched with shared edges, yet by construction, we know at least one triangular cell is matched with a free edge, as free edges always exist in some $\tau_{t}$. Finally, $\sigma_{0}\prec \tau_{t}$ cannot be edge on $\pat_{p}$ , this case is obvious as such $\sigma_{0}$ can only be matched with some $\tau_{s}$ on the boundary, according to our scheme; the only chance to get back is having some $\tau_{t}$ in the closed path, which is impossible. 

What remains to be proved is that $\sigma_{0}$ can not be any interior edge of $\pat_{p}$, for any $p$. By the inductive construction, we can draw a path connecting the sequence $\sigma_{0},\cdots,\sigma_{t}$ all the way to an edge on the boundary, this contradicts with the path being closed. We are done.  
\end{appendices}
\bibliographystyle{plain} 
\bibliography{patrol}
\end{document}